%% file: variational.tex
\documentclass[11pt]{article}
\usepackage[includeheadfoot,margin=1in]{geometry}
\usepackage{times}
\usepackage{cite}
\usepackage{amsmath,amsfonts,amssymb,amsthm}
\usepackage{mathtools}
\usepackage{enumerate}
\usepackage{verbatim}
\usepackage{commath}
    \usepackage[usenames]{color}
    \definecolor{plum}  {rgb}{.4,0,.4}
    \definecolor{BrickRed} {rgb}{0.6,0,0}
	\definecolor{DarkBlue} {rgb}{0,0,0.6}
\usepackage[plainpages=false,pdfpagelabels,colorlinks=true,linkcolor=BrickRed,citecolor=plum,urlcolor=DarkBlue]{hyperref}
\usepackage[mathcal]{euscript}
\usepackage{cleveref}

\def\ddefloop#1{\ifx\ddefloop#1\else\ddef{#1}\expandafter\ddefloop\fi}

\def\ddef#1{\expandafter\def\csname b#1\endcsname{\ensuremath{\boldsymbol{#1}}}}
\ddefloop ABCDEFGHIJKLMNOPQRSTUVWXYZabcdefghijklmnopqrstuvwxyz\ddefloop

\def\ddef#1{\expandafter\def\csname c#1\endcsname{\ensuremath{\mathcal{#1}}}}
\ddefloop ABCDEFGHIJKLMNOPQRSTUVWXYZ\ddefloop
 
\def\ddef#1{\expandafter\def\csname s#1\endcsname{\ensuremath{\mathsf{#1}}}}
\ddefloop ABCDEFGHIJKLMNOPQRSTUVWXYZ\ddefloop

\def\Reals{{\mathbb R}}

\def\Ex{{\mathbf E}} 

\def\trn{{\mathsf T}} 
\def\deq{:=}
\def\E{\Ex}

\makeatletter
\newsavebox{\@brx}
\newcommand{\llangle}[1][]{\savebox{\@brx}{\(\m@th{#1\langle}\)}%
  \mathopen{\copy\@brx\kern-0.5\wd\@brx\usebox{\@brx}}}
\newcommand{\rrangle}[1][]{\savebox{\@brx}{\(\m@th{#1\rangle}\)}%
  \mathclose{\copy\@brx\kern-0.5\wd\@brx\usebox{\@brx}}}
\makeatother

\newtheorem{theorem}{Theorem}

\newtheorem{proposition}{Proposition}

\begin{document}

\title{A Variational Approach to Sampling in Diffusion Processes}

\author{Maxim Raginsky\thanks{University of Illinois, Urbana, IL 61801, USA}\\ \href{mailto:maxim@illinois.edu}{maxim@illinois.edu}}
\date{}
\maketitle

\centerline{\small\textit{Dedicated to the memory of Sanjoy K.~Mitter (1933-2023)}}

\smallskip

\begin{abstract} We revisit the work of Mitter and Newton on an information-theoretic interpretation of Bayes' formula through the Gibbs variational principle. This formulation allowed them to pose nonlinear estimation for diffusion processes as a problem in stochastic optimal control, so that the posterior density of the signal given the observation path could be sampled by adding a drift to the signal process. We show that this control-theoretic approach to sampling provides a common mechanism underlying several distinct problems involving diffusion processes, specifically importance sampling using Feynman--Kac averages, time reversal, and Schr\"odinger bridges. \end{abstract}

\input{variational_body.tex}

\bibliographystyle{siamplain}
\bibliography{variational.bib}

\end{document}

%% file: variational_body.tex

\section{Introduction}
\label{sec:intro}

In a remarkable paper \cite{Mitter_Newton}, Mitter and Newton showed that the Kallianpur--Striebel formula, a core ingredient in the theory of nonlinear filtering for diffusion processes, can be derived from the Gibbs variational principle pertaining to the minimization of a certain free energy functional on the space of probability measures over paths. Among other things, this variational formulation provided an information-theoretic explanation of the fact that the PDE for the logarithm of the filtering density, i.e., the posterior density of the signal process given the observation path, has the form of the Hamilton--Jacobi--Bellman equation for the value function of a particular stochastic control problem, a coincidence that had been noted earlier in several works \cite{Mitter_QM,Mitter_1981,Fleming_1982}. Moreover, using this control-theoretic interpretation together with Girsanov theory, Mitter and Newton have shown that one can obtain exact samples from the filtering density by the addition of a drift term to the signal process, where the drift is of the state feedback form and is equal to the negative gradient of the value function. 

In this article, we show that a variational formulation based on free energy minimization underlies a broad circle of questions pertaining to diffusion processes which include, in addition to path estimation, such problems as time reversal \cite{Anderson_1982,Haussmann_Pardoux,Follmer_1988}, the Schr\"odinger bridge problem \cite{DaiPra,Follmer_1988}, and importance sampling via Feynman--Kac averages \cite{Ezawa_Klauder_Shepp}. In fact, this variational interpretation was implicit in some of the existing treatments of these problems; our aim here is to provide a unifying perspective and to draw attention to the fact that the particular constructions that emerge in the solutions of these problems can all be viewed as instances of stochastic optimal control of diffusion processes, in the spirit of the original work of Mitter and Newton.

The remainder of the paper is structured as follows. In \Cref{sec:Gibbs} we review the Gibbs variational principle in a general setting and discuss the problem of generating samples from Gibbs measures. Diffusion processes are introduced in \Cref{sec:setup}, followed by the formulation of sampling as a problem in stochastic optimal control in \Cref{sec:control}. Next, in \Cref{sec:consequences} we revisit several problems involving sampling in diffusion processes and illustrate that they can be viewed as instances of free energy minimization. Since we are building on the ideas of \cite{Mitter_Newton}, we adopt a great deal of their notation and set-up.

\section{The Gibbs variational principle}
\label{sec:Gibbs}

Let $(\sX,\cX)$ be a standard Borel space. Let $\cP(\sX)$ and $\cH(\sX)$ be the space of all probability measures on $(\sX,\cX)$ and the space of all measurable functions $H : \sX \to (-\infty,+\infty]$, respectively. Let $P,\tilde{P} \in \cP(\sX)$ and $H \in \cH(\sX)$ be given. Then we define the following quantities:
\begin{align}
	\begin{split}\label{eq:KL}
	D(\tilde{P}\|P) &\deq \int_\sX \log \left(\frac{\dif \tilde{P}}{\dif P}\right) \dif \tilde{P} \text{ if } \tilde{P} \ll P, \\
	& \qquad +\infty \text{ otherwise},
	\end{split}\\
	\begin{split}\label{eq:EqFF}
	i(H) &\deq - \log \left(\int_\sX \exp(-H) \dif P \right) \text{ if } 0 < \int_\cX \exp(-H) \dif P < \infty, \\
	& \qquad -\infty \text{ otherwise},
	\end{split}\\
	\begin{split}\label{eq:energy}
	\langle H, \tilde{P} \rangle &\deq \int_\sX H \dif\tilde{P} \text{ if the integral is finite}, \\
	& \qquad +\infty \text{ otherwise}.
	\end{split}
\end{align}
The quantity in \cref{eq:KL} is the relative entropy of $P$ with respect to $\tilde{P}$. In the context of statistical physics, $\sX$ acquires the interpretation of the state (or configuration) space of some physical system, $P$ is some base (or reference) probability measure on the state space, and $H$ is the energy (or Hamiltonian) function. Under this interpretation, the quantity $i(H)$ defined in \cref{eq:EqFF} is the equilibrium free energy (at unit temperature), while the quantity $\langle H, \tilde{P} \rangle$ is the average energy under an alternative probability measure $\tilde{P}$. The \textit{Gibbs variational principle} \cite{Dupuis_Ellis,Mitter_Newton} states that, under some regularity conditions, $i(H)$ is the minimum value of the \textit{free energy}
\begin{align}\label{eq:free_energy}
	F(\tilde{P}) \deq \langle H, \tilde{P} \rangle + D(\tilde{P} \| P)
\end{align}
among all $\tilde{P}$, and characterizes the unique minimizer of $F(\cdot)$ that attains $i(H)$:

\begin{proposition}\label{prop:Gibbs} Let $P \in \cP(\sX)$ and $H \in \cH(\sX)$ be such that
	\begin{align*}
		-\int_\sX H \exp(-H) \dif P < \infty,
	\end{align*}
	with the convention $+\infty \cdot \exp(-\infty) = 0$. Then the probability measure 
	 $P^* \in \cP(\sX)$ defined by
\begin{align*}
	\frac{\dif P^*}{\dif P} = \frac{\exp(-H)}{\int_\sX \exp(-H) \dif P}
\end{align*}
is the unique minimizer of \cref{eq:free_energy}, and
\begin{align*}
	i(H) = F(P^*) = \min_{\tilde{P} \in \cP(\sX)} F(\tilde{P}).
\end{align*}
\end{proposition}
In many cases, $(\sX,\cX)$ has the product structure $(\sX_0 \times \bar{\sX}, \cX_0 \otimes \bar{\cX})$, and we are interested in minimizing the free energy $F(\cdot)$ subject to constraints on the marginal probability law of $X_0$, where $X$ splits into $(X_0,\bar{X})$. That is, we disintegrate the reference measure $P$ as $P(\cdot) = \int_{\sX_0} \mu(\dif x_0) P^{x_0}(\cdot)$, where $\mu$ is the marginal law of $X_0$ under $P$ and $P^{x_0}$ is the regular conditional probability law of $X$ given $X_0 = x_0$, and then minimize the free energy $F(\tilde{P})$ over all $\tilde{P} \in \cP(\sX_0 \times \bar{\sX})$, possibly subject to a constraint of the form $\tilde{\mu} \in \cC$, where $\tilde{\mu}$ is the marginal law of $X_0$ under $\tilde{P}$ and $\cC$ is some subset of $\cP(\sX_0)$. The main idea is to use the chain rule for the relative entropy, which says
\begin{align*}
	D(\tilde{P}\|P) = D(\tilde{\mu} \| \mu) + \int_{\sX_0} \tilde{\mu}(\dif x_0) D(\tilde{P}^{x_0} \| P^{x_0}),
\end{align*}
and Fubini's theorem to decompose the free energy as
\begin{align}
	F(\tilde{P}) &= \langle H, \tilde{P} \rangle + D( \tilde{P} \| P) \nonumber\\
	&= D( \tilde{\mu} \| \mu ) + \int_{\sX_0} \tilde{\mu}(\dif x_0) \left( \int_{\bar{\sX}} \tilde{P}^{x_0}(\dif \bar{x})  H(x_0,\bar{x}) + D(\tilde{P}^{x_0} \| P^{x_0}) \right)  \nonumber\\
	&= D( \tilde{\mu} \| \mu) + \int_{\sX_0} \tilde{\mu}(\dif x_0) F(x_0, \tilde{P}^{x_0}), \label{eq:free_energy_decomposition}
\end{align}
where we have denoted by $F(x_0,\tilde{P}^{x_0})$ the free energy of $\tilde{P}^{x_0}$ with respect to the energy function $H(x_0,\cdot)$ and the reference measure $P^{x_0}$. Then, applying \Cref{prop:Gibbs} conditionally on $x_0$, we have
\begin{align*}
	F(\tilde{P}) &\ge \langle v, \tilde{\mu} \rangle + D(\tilde{\mu} \| \mu),
\end{align*}
where 
\begin{align*}
	v(x_0) \deq - \log \int_{\bar{\sX}} e^{-H(x_0,\cdot)} \dif P^{x_0}
\end{align*}
is the minimum value of $F(x_0,\cdot)$ achieved uniquely by $\dif P^{*,x_0} \propto \exp{-H(x_0,\cdot)}\dif P^{x_0}$. We conclude that 
\begin{align*}
	\min_{\tilde{P} \in \cP(\sX):\, \tilde{\mu} \in \cC} F(\tilde{P}) = \min_{\tilde{\mu} \in \cC} \left( \langle v, \tilde{\mu} \rangle + D(\tilde{\mu} \| \mu) \right).
\end{align*}
In particular, if $\cC = \{\mu\}$, then the minimum is attained uniquely by the Gibbs mixture
\begin{align*}
	P^{*,\mu}(\cdot) \deq \int_{\sX_0} \mu(\dif x_0) P^{*,x_0}(\cdot).
\end{align*}
If, on the other hand, $\cC = \cP(\sX_0)$ (i.e., the marginal $\tilde{\mu}$ is completely unconstrained), then we apply the Gibbs variational principle again to get
\begin{align*}
	\min_{\tilde{P} \in \cP(\sX)} F(\tilde{P}) &= \min_{\tilde{\mu} \in \cP(\sX_0)} \left( \langle v, \tilde{\mu} \rangle + D(\tilde{\mu} \| \mu) \right) \\
	&= - \log \int_{\sX_0} e^{-v(x_0)}\mu(\dif x_0) \\
	&= - \log \int_{\sX_0 \times \bar{\sX}} e^{-H(x_0,\bar{x})} P^{x_0}(\dif \bar{x})\mu(\dif x_0) \\
	&= - \log \int_\sX e^{-H(x)} P(\dif x),
\end{align*}
where the minimum is achieved uniquely by the mixture $P^{*,\mu_*}$ with
\begin{align*}
	\frac{\dif \mu_*}{\dif \mu} &= \frac{\exp(-v)}{\int_{\sX_0} \exp(-v)\dif\mu}.
\end{align*}
The problem of interest is how to generate samples from $P^*$ (or from $P^{*,\mu_*}$) when we have the means to generate samples from the reference measure $P$. In general, this runs into issues of computational tractability. However, in some specific instances it may be possible to exploit additional structure of the problem to deduce the existence of a measurable mapping $\Phi : \sX \to \sX$, such that $P^* = P \circ \Phi^{-1}$---in other words, we first obtain a sample $X$ from $P$ and then transform it into a sample $\Phi(X)$ from $P^*$. As we shall see next, this is indeed possible when $P$ is a sufficiently regular probability law of a diffusion process.

\section{The problem set-up}
\label{sec:setup}

We now particularize the setting of \Cref{sec:Gibbs} to the case when $(\sX,\cX)$ is the space $(C([0,T]; \Reals^n),\cB_T)$ of continuous paths $x \colon [0,T] \to \Reals^n$, where $\cB_T$ is the Borel $\sigma$-algebra induced by the uniform norm topology. The reference measure $P$ is the probability law of the diffusion process governed by the It\^o integral equation
\begin{align}
	\begin{split}\label{eq:ref_process}
	X_t &= X_0 + \int^t_0 b(X_s,s) \dif s + \int^t_0 \sigma(X_s,s) \dif W_s, \qquad 0 \le t \le T \\
	X_0 &\sim \mu,
	\end{split}
\end{align}
where $X_t$ takes values in $\Reals^n$ and $W_t$ takes values in $\Reals^m$. Here, the Borel probability law $\mu$ and the mappings $b(\cdot,\cdot)$ and $\sigma(\cdot,\cdot)$ are assumed to satisfy enough regularity conditions for \cref{eq:ref_process} to have a unique strong solution. In the latter case, we will have a filtered probability space $(\Omega,\cF,(\cF_t),{\mathbb P})$ that carries an $\Reals^n$-valued random variable $X_0$ and an $m$-dimensional standard Brownian motion process $W$ independent of $X_0$, as well as a measurable map $\Phi \colon \Reals^n \times C([0,T];\Reals^m) \to \sX$, such that $(X_t = \Phi_t(X_0,W) ; 0 \le t \le T)$ is an $(\cF_t)$-adapted semimartingale satisfying \cref{eq:ref_process}. The following conditions, imposed in \cite{Mitter_Newton}, suffice for our purposes as well:
\begin{itemize}
	\item[(R1)] there exists a constant $c > 0$, such that
	\begin{align*}
		\int_{\Reals^n} \exp(c|z|^2)\mu(\dif z) < \infty;
	\end{align*}
	\item[(R2)] there exists a constant $K > 0$, such that $b$ and $\sigma$ satisfy
	\begin{align*}
		& |b(x,t)-b(\bar{x},t)| + |\sigma(x,t)-\sigma(\bar{x},t)| \le K|x - \bar{x}|,  \\
		& |b(x,t)| \le K(1+|x|), \\
		& |\sigma(x,t)| \le K
	\end{align*}
for all $x,\bar{x} \in \Reals^n$ and all $0 \le t \le T$, where we use $|\cdot|$ to denote the Euclidean norm for vectors and the Hilbert--Schmidt norm for matrices;
\end{itemize}
Next, we assume that the Hamiltonian function $H \colon \sX \to \Reals$ is of the form
\begin{align}\label{eq:Hamiltonian}
	H(X) = \int^T_0 f(X_t,t)\dif t + g(X_T),
\end{align}
and the functions $f(\cdot,\cdot)$ and $g(\cdot)$ satisfy the following:
\begin{itemize}
	\item[(H1)] $f$ and $g$ are bounded from below, continuously differentiable, and there exists a constant $C > 0$, such that
	\begin{align*}
		|f(0,t)| &\le C, \\
		\sum^n_{i=1} \Big|\frac{\partial}{\partial x_i}f(0,t)\Big| &\le C
	\end{align*}
	for all $0 \le t \le T$;
	\item[(H2)] the derivatives of $f$ and $g$ are Lipschitz continuous: there exists a constant $M > 0$, such that
	\begin{align*}
		\sum^n_{i=1} \Big|\frac{\partial}{\partial x_i}f(x,t) - \frac{\partial}{\partial x_i}f(\bar{x},t)\Big| & \le M |x - \bar{x}|, \\
		\sum^n_{i=1} \Big|\frac{\partial}{\partial x_i}g(x) - \frac{\partial}{\partial x_i}g(\bar{x})\Big| & \le M |x - \bar{x}|
 	\end{align*}
	for all $x,\bar{x} \in \Reals^n$ and all $0 \le t \le T$.
	\end{itemize}
The above assumptions ensure that
\begin{align*}
	\E [\exp(-H(X))] < \infty, \,\, - \E [H(X)\exp(-H(X))] \le \E[\exp(-2H(X))] < \infty
\end{align*}
where $\E[\cdot]$ denotes expectation with respect to $P$ \cite{Mitter_Newton}. It is also straightforward to verify that, under (H1) and (H2), $f$ and $g$ are of at most quadratic growth in $x$ and their derivatives are of at most linear growth in $x$, uniformly in $0 \le t \le T$.

We will need to consider the case when the initial condition $X_0$ is nonrandom, i.e., $\mu$ in \cref{eq:ref_process} is a Dirac measure centered at some $z \in \Reals^n$. It is convenient, just as in \cite{Mitter_Newton}, to define for each $z \in \Reals^n$ and each $0 \le s \le T$ the process $(X^{z,s}_t : s \le t \le T)$ as the solution of \eqref{eq:ref_process} on the time interval $s \le t \le T$ with initial condition $X^{z,s}_s = z$. Then we will denote by $P^z$ the probability law of the process $X^{z,0}$. We also define the measurable maps
\begin{align}\label{eq:partial_H}
	H \colon [0,T] \times \Reals^n \times \sX \to \Reals, \qquad 
	H(s,z,X^{z,s}) \deq \int^T_s f(X^{z,s}_t,t) \dif t + g(X^{z,s}_T)
\end{align}
[so that, in particular, $H(X^{z,0}) = H(0,z,X^{z,0})$] and
\begin{align}\label{eq:value_function}
	v \colon \Reals^n \times [0,T] \to \Reals, \qquad 
	v(z,s) \deq - \log \E \exp(-H(s,z,X^{z,s})),
\end{align}
which is the equilibrium free energy of $P^{z,s}$ corresponding to the Hamiltonian function $H(s,z,\cdot)$.

\section{The optimal control problem}
\label{sec:control}

We now consider the controlled equation
\begin{align}\label{eq:controlled_process}
	\tilde{X}_t = z + \int^t_0 \Big(b(\tilde{X}_s,s) + a(\tilde{X}_s,s)u(\tilde{X}_s,s)\Big) \dif s + \int^t_0 \sigma(\tilde{X}_s,s)\dif \tilde{W}_s,
\end{align}
where $a(x,s) \deq \sigma(x,s)\sigma(x,s)^\trn$, and where the measurable function $u \colon \Reals^n \times [0,T] \to \Reals^n$ is the control. Let $\sU$ denote the set of all  $u$ with the following properties:
\begin{itemize}
	\item[(U1)] $u$ is continuous;
	\item[(U2)] $\E Z^u = 1$, where
	\begin{align}\label{eq:att_density}
		Z^u \deq \exp\left(\int^T_0 u^\trn \sigma(X^{z,0}_t,t)\dif W_t - \frac{1}{2}\int^T_0 |\sigma^\trn u(X^{z,0}_t,t)|^2 \dif t\right),
	\end{align}
\end{itemize}
where the objects $(\Omega,\cF,(\cF_t),{\mathbb P})$, $W$, and $X^{z,0}$ are as defined above. We then have the following:

\begin{proposition}[Mitter--Newton] If $b$ and $\sigma$ satisfy (R2) and if $u \in \sU$, then the controlled equation \eqref{eq:controlled_process} has a weak solution which is unique in probability law.
\end{proposition}

Following Bene\v{s} \cite{Benes}, we will refer to $u \in \sU$ as \textit{admissible controls}, the maps $(x,s) \mapsto b(x,s) + a(x,s)u(x,s)$ as \textit{admissible drifts}, and to $Z^u$ in \cref{eq:att_density} as \textit{attainable densities}. The problem, then, is to show that the Gibbs density
\begin{align}\label{eq:Gibbs_density}
	\frac{\dif P^{*,z}}{\dif P^z} = \frac{\exp(-H)}{\int_\sX \exp(-H)\dif P^z} = \frac{\exp(-H(0,z,\cdot))}{\int_\sX \exp(-H(0,z,\cdot))\dif P^z}.
\end{align}
is an attainable density and to identify the admissible control $u_* \in \sU$, such that
\begin{align*}
	\frac{\dif P^{*,z}}{\dif P^z} = Z^{u_*}.
\end{align*}
Let $(\tilde{\Omega},\tilde{\cF},(\tilde{\cF}_t),\tilde{{\mathbb P}},\tilde{X},\tilde{W})$ be a weak solution of \cref{eq:controlled_process} corresponding to an admissible control $u \in \sU$. Let $\tilde{P}$ denote the probability law of $\tilde{X}$. We then define the control cost of $u$ by
\begin{align}\label{eq:control_cost}
	\begin{split}
	J(u,z) &\deq  \langle H, \tilde{P} \rangle + D(\tilde{P}\|P^z) \\
	&= \tilde{\E}\left[ \int^T_0 \Big(f(\tilde{X}_t,t)  +  \frac{1}{2} |\sigma^\trn u(\tilde{X}_t,t)|^2\Big) \dif t +  g(\tilde{X}_T)\right].
	\end{split}
\end{align}
The optimal control that attains the minimum of $J(u,z)$ is given by the following theorem, in the spirit of Theorem~4.2 in \cite{Mitter_Newton}:
\begin{theorem}\label{thm:main} Suppose that $b$, $\sigma$, $f$, and $g$ satisfy (R2), (H1), (H2). Define the function $u_* \colon \Reals^n \times [0,T] \to \Reals^n$ by
	\begin{align}\label{eq:ustar}
		u_*(x,t) \deq - \left(\frac{\partial v}{\partial x}(x,t)\right)^\trn,
	\end{align}
	where $v$ is defined in \cref{eq:value_function}. Then $u_*$ is an admissible control, and for all $z \in \Reals^n$ and all $\tilde{P} \in \cP(\sX)$ (not necessarily arising from an admissible control),
	\begin{align}\label{eq:ustar_optimality}
		J(u_*,z) \le D(\tilde{P} \| P^z) + \langle H(0,z,\cdot), \tilde{P}_X \rangle.
	\end{align}
\end{theorem}
\begin{proof}
	
	We follow the overall logic of the proof of Theorem~4.2 in \cite{Mitter_Newton}, but use slightly different techniques in order to adapt the argument to our assumptions and to make it relatively self-contained. 
		
\paragraph{Stochastic flow estimates}	First, from the theory of stochastic flows \cite[Ch.~4]{Kunita_book} it follows that, for all $0 \le s \le t \le T$, the map $z \mapsto X^{z,s}_t$ is differentiable, and the Jacobian $\Psi^{z,s}_t \deq \frac{\partial}{\partial z}X^{z,s}_t$, taking values in $\Reals^{n \times n}$, satisfies
	\begin{align}\label{eq:variational_SDE}
		\Psi^{z,s}_t = I + \int^t_s \frac{\partial b}{\partial x}(X^{z,s}_r,r) \Psi^{z,s}_r \dif r + \sum^m_{i=1} \int^t_s \frac{\partial \sigma_i}{\partial x}(X^{z,s}_r,r) \Psi^{z,s}_r \dif W^i_r,
	\end{align}
	where $\sigma_i(\cdot,\cdot)$ denotes the $i$th column of $\sigma(\cdot,\cdot)$, $W^i$ denotes the $i$th coordinate of $W$,  $\frac{\partial b}{\partial x}(x,t)$ denotes the Jacobian of $b(x,t)$ w.r.t.\ the space variable $x$, etc. Moreover, the following moment estimates hold for each $p \ge 1$ uniformly in $s \le t \le T$:
	\begin{align}
		\E |X^{z,s}_t|^{p} &\le C_{T,p}(1+|z|^{p}),  \label{eq:X_moment}\\
		\E \| \Psi^{z,s}_t \|^{p} & \le C_{T,p}, \label{eq:Psi_moment}
	\end{align}
where $C_{T,p} < \infty$ is a constant not depending on $z$ or $s$. For any $z,\bar{z} \in \Reals^n$, any $0 \le s \le t \le T$, and any $p \ge 2$, 
\begin{align*}
& \E \sup_{s \le r \le t} |X^{\bar{z},s}_r - X^{z,s}_r|^p \\
&\qquad \le (m+2)^{p-1} \Bigg(|\bar{z}-z|^p + \E \sup_{s \le r \le t} \Bigg|\int^r_s \big(b(X^{\bar{z},s}_\tau,\tau) - b(X^{z,s}_\tau,\tau)\big) \dif \tau\Bigg|^p \\
& \qquad \qquad \qquad + \sum^m_{i=1}\E \sup_{s \le r \le t}\Bigg| \int^r_s \big(\sigma_i(X^{\bar{z},s}_\tau,\tau) - \sigma_i(X^{z,s}_\tau,\tau)\big)\dif W^i_\tau \Bigg|^p \Bigg) \\
& \qquad \le (m+2)^{p-1} \Bigg(|\bar{z}-z|^p + \E  \Bigg|\int^t_s \big(b(X^{\bar{z},s}_\tau,\tau) - b(X^{z,s}_\tau,\tau)\big) \dif \tau\Bigg|^p \\
& \qquad \qquad \qquad + \sum^m_{i=1}\E\Bigg| \int^t_s \big(\sigma_i(X^{\bar{z},s}_\tau,\tau) - \sigma_i(X^{z,s}_\tau,\tau)\big)\dif W^i_\tau \Bigg|^p \Bigg),
\end{align*}
where we have used Jensen's inequality and Doob's submartingale inequality. Using (R2), H\"older's inequality, and the Burkholder--Davis--Gundy inequality, we have
\begin{align*}
 \E  \Bigg|\int^t_s \big(b(X^{\bar{z},s}_\tau,\tau) - b(X^{z,s}_\tau,\tau)\big) \dif \tau\Bigg|^p \le K^p T^{p-1} \int^t_s \sup_{s \le r \le \tau} |X^{\bar{z},s}_r - X^{z,s}_r|^p \dif \tau
\end{align*}
and
\begin{align*}
	\E\Bigg| \int^t_s \big(\sigma_i(X^{\bar{z},s}_\tau,\tau) - \sigma(X^{z,s}_\tau,\tau)\big)\dif W^i_\tau \Bigg|^p \le K^p T^{p-\frac{1}{2}} \int^t_s \E \sup_{s \le r \le \tau} |X^{\bar{z},s}_r - X^{z,s}_\tau|^p \dif \tau.
\end{align*}
Putting all of this together and using Gr\"onwall's lemma, we get
\begin{align}\label{eq:flow_continuity}
	\E \sup_{s \le t \le T} |X^{\bar{z},s}_t - X^{z,s}_t|^p \le \tilde{C}_{T,K,p} |\bar{z}-z|^p \qquad \text{for all } z,\bar{z},s.
\end{align}
A similar argument leads to the estimate
\begin{align}\label{eq:flow_moments}
	\E \sup_{s \le t \le T} |X^{z,s}_t|^p \le \tilde{C}_{T,K,p} (1+|z|^p) \qquad \text{for all } z,s.
\end{align}
In the above, $\tilde{C}_{T,K,p} < \infty$ is a constant that does not depend on $z,\bar{z},s$.

We now apply the mean-value theorem to write
\begin{align*}
	X^{\bar{z},s}_t - X^{z,s}_t &= \bar{z} - z + \int^t_s \Bigg(\int^1_0 \frac{\dif}{\dif v} b(X^{z,s}_\tau + v(X^{\bar{z},s}_\tau - X^{z,s}_\tau),\tau)\dif v\Bigg) \dif \tau  \\
	& \quad + \sum^m_{i=1} \int^t_s \Bigg(\int^1_0 \frac{\dif}{\dif v} \sigma_i(X^{z,s}_\tau + v(X^{\bar{z},s}_\tau - X^{z,s}_\tau),\tau) \dif v\Bigg) \dif W^i_\tau \\
	&= \bar{z} - z + \int^t_s \Bigg(\int^1_0 \frac{\partial b}{\partial x}(X^{z,s}_\tau + v (X^{\bar{z},s}_\tau - X^{z,s}_\tau),\tau) \dif v\Bigg) (X^{\bar{z},s}_\tau - X^{z,s}_\tau) \dif \tau \\
	& \quad + \sum^m_{i=1} \int^t_s \Bigg(\int^1_0 \frac{\partial \sigma_i}{\partial x}(X^{z,s}_\tau + v(X^{\bar{z},s}_\tau - X^{z,s}_\tau),\tau)\dif v\Bigg) (X^{\bar{z},s}_\tau - X^{z,s}_\tau) \dif W^i_\tau.
\end{align*}
Let $A$ be an arbitrary compact subset of $\Reals^n$. Then, with the help of \cref{eq:variational_SDE}, \cref{eq:flow_continuity}, \cref{eq:flow_moments}, and (R2),  similar reasoning as the one used for proving \cref{eq:controlled_process} shows that the inequality
\begin{align}\label{eq:flow_diff}
	\E \sup_{s \le t \le T} |X^{\bar{z},s}_t - X^{z,s}_t - \Psi^{z,s}_t (\bar{z}-z)|^p = O(|\bar{z}-z|^{2p})
\end{align}
holds uniformly on $A \times [0,T]$. Next, let
\begin{align*}
	\xi(z,s) \deq \int^T_s \frac{\partial f}{\partial x}(X^{z,s}_t,t) \Psi^{z,s}_t \dif t + \frac{\partial g}{\partial x}(X^{z,s}_T) \Psi^{z,s}_T.
\end{align*}
Then, using the mean-value theorem again, we have
\begin{align*}
&	\E |H(s,\bar{z},X^{\bar{z},s}) - H(s,z,X^{z,s}) - \xi(z,s)(\bar{z}-z)|^p \\
& \le 2^{p-1} \Bigg(\E\Bigg| \int^T_s \Bigg(\int^1_0 \frac{\dif}{\dif v}f(X^{z,s}_t + v(X^{\bar{z},s}_t - X^{z,s}_t),t)\dif v\Bigg) (X^{\bar{z},s}_t - X^{z,s}_t)\dif t \\
& \qquad \qquad \qquad \qquad - \int^T_s \frac{\partial f}{\partial x}(X^{z,s}_t,t) \Psi^{z,s}_t \dif t\Bigg|^p \\
& \qquad \qquad + \E\Bigg|\int^1_0 \frac{\dif}{\dif v} g(X^{z,s}_T + v(X^{\bar{z},s}_T - X^{z,s}_T)) \dif v -  \frac{\partial g}{\partial x}(X^{z,s}_T) \Psi^{z,s}_T \Bigg|^p\Bigg).
\end{align*}
Using H\"older's inequality, (H1), (H2), and \cref{eq:flow_continuity,eq:flow_moments,eq:flow_diff}, 
\begin{align*}
	& \E\Bigg| \int^T_s \Bigg(\int^1_0 \frac{\dif}{\dif v}f(X^{z,s}_t + v(X^{\bar{z},s}_t - X^{z,s}_t),t)\dif v\Bigg) (X^{\bar{z},s}_t - X^{z,s}_t)\dif t  \\
	& \qquad \qquad \qquad \qquad - \int^T_s \frac{\partial f}{\partial x}(X^{z,s}_t,t) \Psi^{z,s}_t \dif t\Bigg|^p = O(|\bar{z}-z|^{2p})
\end{align*}
and
\begin{align*}
	\E\Bigg|\int^1_0 \frac{\dif}{\dif v} g(X^{z,s}_T + v(X^{\bar{z},s}_T - X^{z,s}_T)) \dif v -  \frac{\partial g}{\partial x}(X^{z,s}_T) \Psi^{z,s}_T \Bigg|^p = O(|\bar{z}-z|^{2p}),
\end{align*}
both holding uniformly on $A \times [0,T]$. Consequently,
\begin{align*}
	\E|H(s,\bar{z},X^{\bar{z},s})-H(s,z,X^{z,s})-\xi(z,s)(\bar{z}-z)|^p  = O(|\bar{z}-z|^{2p}).
\end{align*}
Thus, using the chain rule, (H1), and (H2), we conclude that
\begin{align*}
	\E|\Theta(\bar{z},s)-\Theta(z,s)-\xi(z,s)\Theta(z,s)(\bar{z}-z)|^p = o(|\bar{z}-z|^p)
\end{align*}
holds uniformly on $A \times [0,T]$, where
\begin{align*}
	\Theta(z,s) \deq \exp(-H(s,z,X^{z,s})).
\end{align*}
This shows that $\frac{\partial}{\partial z}\rho(z,s) = \E \xi(z,s)\Theta(z,s)$, where $\rho(z,s) \deq \E \Theta(z,s)$. Since
\begin{align*}
	\inf_{z \in A} \inf_{s \in [0,T]} \rho(z,s) \ge \inf_{z \in A} \inf_{s \in [0,T]} \exp \E \log \Theta(z,s) > 0,
\end{align*}
we arrive at the following stochastic representation of $u_*(\cdot,\cdot)$:
\begin{align*}
	u_*(z,s) = \frac{\E \xi(z,s)\Theta(z,s)}{\E \Theta(z,s)}.
\end{align*}

\paragraph{Analysis in a special case} Next, just as in \cite{Mitter_Newton}, we consider the special case when $b, f, g$ are bounded and there exists a constant $c > 0$, such that
\begin{align}\label{eq:uniform_ellipticity}
	z^\trn a(\tilde{z},t) z \ge c|z|^2, \qquad \forall z,\tilde{z} \in \Reals^n,\, t \in [0,T].
\end{align}
Then, by the Feynman--Kac formula \cite{Kallenberg_probability}, $\rho$ is a $C^{2,1}$ solution of the PDE
\begin{align*}
	\frac{\partial}{\partial t}\rho(z,t) + \cL \rho(z,t) - f(z,t)\rho(z,t) = 0 \text{ on } \Reals^n \times [0,T], \qquad \rho(z,T) = \exp(-g(z))
\end{align*}
where 
\begin{align*}
	\cL = \sum_i b_i \frac{\partial}{\partial z_i} + \frac{1}{2}\sum_{i,j} a_{ij} \frac{\partial^2}{\partial z_i \partial z_j}
\end{align*}
is the generator of the reference process \cref{eq:ref_process}. The function $v = -\log \rho$ satisfies the nonlinear PDE
\begin{align}\label{eq:HJB}
	\frac{\partial}{\partial t}v(z,t) + \cL v(z,t) + f(z,t) = \frac{1}{2} \frac{\partial v}{\partial z}(z,t) a(z,t) \left(\frac{\partial v}{\partial z}(z,t)\right)^\trn
\end{align}
on $\Reals^n \times [0,T]$, subject to the terminal condition $v(z,T) = g(z)$. The quantity on the right-hand side of \cref{eq:HJB} arises from the identity
\begin{align*}
	\min_{u \in \Reals^n} \Bigg\{ u^\trn a(z,t) \left(\frac{\partial v}{\partial z}(z,t)\right)^\trn + u^\trn a(z,t) u \Bigg\} = -  \frac{1}{2}\frac{\partial v}{\partial z}(z,t) a(z,t) \left(\frac{\partial v}{\partial z}(z,t)\right)^\trn,
\end{align*}
where, since $a(z,t)$ is positive definite by \eqref{eq:uniform_ellipticity}, the minimum is achieved uniquely by $u_*(z,t) = - \big(\frac{\partial v}{\partial z}(z,t)\big)^\trn$; this observation allows us to interpret \cref{eq:HJB} as the Hamilton--Jacobi--Bellman PDE for the optimal stochastic control problem of minimizing the expected cost in \cref{eq:control_cost} and $v$ as the corresponding value function \cite{Fleming_1975}.

Now, it follows from \cref{eq:X_moment,eq:Psi_moment} and from the boundedness of $f$, $g$, and their derivatives that $u_*$ is also bounded, so (U2) is satisfied by Novikov's theorem \cite{Kallenberg_probability}. Thus, $u_*$ is an admissible control. Therefore, by Girsanov's theorem,
\begin{align*}
	W^*_t = W_t - \int^t_0 \sigma^\trn u_*(X^{z,0}_s,s) \dif s
\end{align*}
is a standard Brownian motion under the probability measure $P^* \equiv P^{u_*}$, given by
\begin{align*}
	\frac{\dif P^{u_*}}{\dif P^{z}} = Z^{u_*}.
\end{align*}
Consequently, applying It\^o's rule and \cref{eq:HJB}, we have
\begin{align*}
	g(X^{z,0}_T) &= v(z,0) - \int^T_0 \Big( f(X^{z,0}_t,t) - \frac{1}{2}|\sigma^\trn u_*(X^{z,0}_t,t)|^2 \Big) \dif t  - \int^T_0 u_*^\trn \sigma(X^{z,0}_t,t) \dif W_t \\
	&= v(z,0) - \int^T_0 \Big(f(X^{z,0}_t,t) + \frac{1}{2}|\sigma^\trn u_*(X^{z,0}_t,t)|^2\Big) \dif t  - \int^T_0 u_*^\trn \sigma (X^{z,0}_t,t) \dif W^*_t. 
\end{align*}
Now, since $(\Omega, \cF, (\cF_t), P^*, X^{z,0}, W^*)$ is a weak solution of \cref{eq:controlled_process} and since $f$, $g$, and $u_*$ are bounded, we can take expectations w.r.t.\ $P^*$ to get
\begin{align*}
	v(z,0) &= \E^* \Bigg[\int^T_0 \Bigg(f(X^{z,0}_t,t) + \frac{1}{2}|\sigma^\trn u_*(X^{z,0}_t,t)|^2\Bigg)\dif t + g(X^{z,0}_T)\Bigg] \\
	&= J(u_*, z).
\end{align*}
Since $v(z,0)$ is the equilibrium free energy of $P^z$ corresponding to the Hamiltonian function $H(\cdot) = H(0,z,\cdot)$ defined in \cref{eq:Hamiltonian}, we have proved \cref{eq:ustar_optimality}. Since $P^*$ is the unique minimizer of the free energy $F(\cdot)$ w.r.t.\ $P^z$, it follows that
\begin{align*}
	Z^{u_*} = \frac{\Theta(z,0)}{\rho(z,0)}
\end{align*}
almost surely.

\paragraph{Approximation} Finally, we remove the additional restrictions on $b$, $\sigma$, $f$, and $g$ and show that this general setting can be reduced to the above special case using an approximation argument. Following \cite{Mitter_Newton}, for any $N = 1,2,\dots$ let
\begin{align*}
	b_N(z,t) &\deq b(z,t)\exp(-|z|^2/N), \\
	f_N(z,t) &\deq f(z,t)\exp(-|z|^2/N), \\
	g_N(z,t) &\deq g(z,t)\exp(-|z|^2/N), \\
	\sigma_N(z,t) &\deq [\begin{matrix} \sigma(z,t) & N^{-1}I_n \end{matrix}] \qquad \text{(an $n \times (m + n)$ matrix)}.
\end{align*}
It follows from (R1), (R2), (H1), (H2) that 
\begin{align*}
	&\sup_{z \in \Reals^n}\sup_{0 \le t \le T}|b_N(z,t)|\\
	&\qquad\le \sup_{z \in \Reals^n}\sup_{0 \le t \le T}|b(z,t)-b(0,t)|e^{-|z|^2/N} + \sup_{z \in \Reals^n}\sup_{0 \le t \le T}|b(0,t)|e^{-|z|^2/N} \\
	&\qquad\le K \sup_{z \in \Reals^n}(|z|e^{-|z|^2/N} + 1) < \infty
\end{align*}
and
\begin{align*}
	\sup_{z \in \Reals^n}\sup_{0 \le t \le T}|f_N(z,t)| &\le \sup_{z \in \Reals^n}\sup_{0 \le t \le T}|f(z,t)|e^{-|z|^2/N} \\
	&\le \sup_{z \in \Reals^n} C(1+|z|)^2 e^{-|z|^2/N} < \infty
\end{align*}
for all $z \in \Reals^n$, and similarly
\begin{align*}
	\sup_{z \in \Reals^n} |g_N(z)| < \infty.
\end{align*}
Thus, $b_N,f_N,g_N$ are bounded and $\sigma_N$ satisfies the uniform ellipticity condition \eqref{eq:uniform_ellipticity} with $a_N(z,t) = \sigma_N(z,t)\sigma_N(z,t)^\trn$. Moreover, $b_N,\sigma_N$ satisfy (R2), $f_N,g_N$ satisfy (H1), (H2) all uniformly in $N$, and $b_N,\sigma_N,f_N,g_N$ and $\frac{\partial b_N}{\partial z},\frac{\partial \sigma_N}{\partial z},\frac{\partial f_N}{\partial z},\frac{\partial g_N}{\partial z}$ converge to $b, [\sigma, 0], f, g$ and to $\frac{\partial b}{\partial z}, [\frac{\partial \sigma}{\partial z },0],\frac{\partial f}{\partial z}, \frac{\partial g}{\partial z}$ respectively, uniformly on compacts. Following \cite{Mitter_Newton}, we will attach a subscript or a superscript $N$ to various processes and random variables defined using $b_N,\sigma_N,f_N,g_N$ instead of $b,\sigma,f,g$, and with $W_t$ replaced with an $(m+n)$-dimensional Brownian motion $(W_t,B_t)$. In particular, for each $N$, we have the processes
$$
X^{N,z,s}_t = z + \int^t_s b_N(X^{N,z,s}_r,r)\dif r + \int^t_s \sigma(X^{N,z,s}_r,r) \dif W_r + N^{-1}(B_t - B_s), \, s \le t \le T
$$
for all $z \in \Reals^n$ and $0 \le s \le T$.  Owing to the properties of $b_N,\sigma_N,f_N,g_N$, the corresponding optimal controls $u_{*N}(\cdot,\cdot)$ satisfy (U1) and (U2). Moreover, the same arguments as the ones used in \cite{Mitter_Newton} show that, for any bounded set $A \subset \Reals^n$,
\begin{align*}
	u_{*N}(z,t) \to u_*(z,t) \qquad \text{as } N \to \infty
\end{align*}
uniformly on $A \times [0,T]$ and
\begin{align*}
	Z_N^{u_{*N}} = \frac{\Theta_N(z,0)}{\rho_N(z,0)} \to \frac{\Theta(z,0)}{\rho(z,0)} = Z^{u_*}, \qquad \text{in probability}.
\end{align*}
Thus, $u_*$ satisfies (U1) and (U2). This proves \eqref{eq:ustar_optimality} in the general case.
\end{proof}

We now turn to the case when the initial condition is random. To that end, consider the controlled process
\begin{align}\label{eq:random_init_cond}
	\begin{split}
		\tilde{X}_t &= \tilde{X}_0 + \int^t_0 \Big(b(\tilde{X}_s,s) + a(\tilde{X}_s,s)u(\tilde{X}_s,s)\Big)\dif s + \int^t_0 \sigma(\tilde{X}_s,s)\dif \tilde{W}_s,\\
		\tilde{X}_0 &\sim \tilde{\mu}
	\end{split}
\end{align}
for $0 \le t \le T$. We thus find ourselves in the situation described at the end of Section~\ref{sec:Gibbs} involving minimization over both the admissible controls $u$ and the initial condition $\tilde{\mu}$, possibly subject to additional constraints on $\tilde{\mu}$. Let $\tilde{P}$ denote the distribution of $\tilde{X}$ for a given pair $(\tilde{\mu},u)$. Then, specializing the decomposition \eqref{eq:free_energy_decomposition} to the present setting, we can write
\begin{align*}
	F(\tilde{P}) = D(\tilde{\mu} \| \mu) + \langle J(u,\cdot), \tilde{\mu} \rangle,
\end{align*}
which is minimized by the choice of $u = u_*$ and of $\tilde{\mu}$ as any minimizer of the functional
\begin{align*}
	\tilde{\mu} \mapsto D(\tilde{\mu} \| \mu) + \langle J(u_*,\cdot), \tilde{\mu} \rangle
\end{align*}
subject to the given constraint $\tilde{\mu} \in \cC$. We will primarily consider the setting when (R1) holds for every element of $\cC$ for some constant $c > 0$.

\section{Consequences}
\label{sec:consequences}

In this section, we will examine several problems pertaining to diffusion processes through the control-theoretic lens of Theorem~\ref{thm:main}. 

\subsection{Feynman--Kac averages} 

Due to the structure of the Hamiltonian function $H$ in \eqref{eq:Hamiltonian}, the Gibbs measures $P^{*,z}$ in \eqref{eq:Gibbs_density} are of the \textit{Feynman--Kac type} \cite{Lorinczi_Gibbs}. Thus, the problem of generating samples from $P^{*,z}$ is synonymous with the problem of computing (or estimating) Feynman--Kac averages of the form
\begin{align}\label{eq:FK_average}
	\langle F, P^{*,z} \rangle = \frac{1}{\int_\sX \exp(-H)\dif P^z} \int_\sX F \exp(-H) \dif P^z
\end{align}
for bounded measurable functions $F : \sX \to \Reals$ on the path space $\sX = C([0,T];\Reals^n)$, provided we have a mechanism for generating random paths under the reference measure $P^z$. In some special cases, a sampling procedure can be built based on the `killing' interpretation of Feynman--Kac averages \cite{Ito_McKean}: If $f$ is everywhere positive and $g \equiv 0$, then we can think of a particle following a path in the $P^z$-ensemble that gets `killed' at a point $(x,t)$ and in time interval $[t,t+\dif t]$ with probability $f(x,t)\dif t$. Then the average of $F$ over the paths that have `survived' at time $T$ is exactly the Feynman--Kac average \eqref{eq:FK_average}. This procedure amounts to a reweighting of sample paths generated according to the reference measure.

To the best of the author's knowledge, the first explicit construction of an alternative sampling method relying, instead of killing, on the addition of a drift to the reference process was given by Ezawa, Klauder, and Shepp \cite{Ezawa_Klauder_Shepp} in the special case when the reference process is an $n$-dimensional Brownian motion starting at $z$ at $t = 0$ and when $g \equiv 0$ in \eqref{eq:Hamiltonian}. In fact, the drift constructed in \cite{Ezawa_Klauder_Shepp} is exactly the optimal control $u_*$ defined in \eqref{eq:ustar}, where the function $\rho = \exp(-v)$ is given as a solution of a certain (linear) PDE. Following the ideas of Fleming \cite{Fleming}, we can view this linear PDE as related to Hamilton--Jacobi--Bellman PDE \eqref{eq:HJB} for the value function $v$ via the logarithmic (or Cole--Hopf) transformation $v = -\log \rho$. No such control-theoretic interpretation was given in the original paper \cite{Ezawa_Klauder_Shepp}, although it was pointed out in later works by other authors \cite{Yasue,Guerra_Morato}.

\subsection{Reciprocal Markov processes and the Schr\"odinger bridge} Let two Borel probability measures $\mu,\mu'$ on $\Reals^n$ be given. Consider the controlled process \eqref{eq:random_init_cond} with $\tilde{X}_0 \sim \mu$. We wish to find an admissible drift $u \in \sU$ such that the `energy'
$$
\frac{1}{2}\tilde{\E}\int^T_0 \frac{1}{2}|\sigma^\trn u(\tilde{X}_t,t)|^2 \dif t
$$
is minimized subject to the constraint $\tilde{X}_T \sim \mu'$. This problem, going back to the work of Sch\"rodinger and Bernstein on so-called reciprocal Markov processes, is now commonly referred to as the \textit{Schr\"odinger bridge} problem \cite{Follmer_1988}. A control-theoretic treatment was given by Dai Pra \cite{DaiPra}. Here, we revisit it from the free energy minimization perspective and, in particular, explicitly identify the corresponding Gibbs measure on the path space.

We assume that, in addition to (R1) and (R2), the uniform ellipticity condition \eqref{eq:uniform_ellipticity} holds, so that the reference process \eqref{eq:ref_process} is nondegenerate, as in \cite{DaiPra}. This ensures the existence of everywhere positive transition densities $p(z,y;s,t)$ for $0 \le s \le t \le T$ and $z,y \in \Reals^n$, so that
\begin{align*}
	\E h(X^{z,s}_t) = \int_{\Reals^n} h(y)p(z,y;s,t) \dif y
\end{align*}
for all bounded, measurable $h : \Reals^n \to \Reals$. We will denote by $\tilde{\mu}'$ the probability law of $X_T$ under the reference process \eqref{eq:ref_process} when $X_0 \sim \mu$, i.e., 
\begin{align*}
	\tilde{\mu}'(A) = \int_A \left(\int_{\Reals^n} p(z,y;0,T)\mu(\dif z)\right) \dif y
\end{align*}
for any Borel set $A \subseteq \Reals^n$; in particular, $\tilde{\mu}'$ has a density with respect to the Lebesgue measure. We assume henceforth that $\mu'$ is absolutely continuous w.r.t.\ $\tilde{\mu}'$.

Following \cite{DaiPra}, we will make essential use of the following key structural result of Beurling \cite{Beurling} and Jamison \cite{Jamison}: Given $\mu$, $\mu'$, and $p$, there exist two unique $\sigma$-finite Borel measures $\nu$ and $\nu'$ on $\Reals^n$, such that the measure
\begin{align}\label{eq:S_coupling_1}
	\pi(E) \deq \int_E p(z,y;0,T)\nu(\dif z)\nu'(\dif y), \qquad E \in \cB(\Reals^n \times \Reals^n)
\end{align}
has marginals $\mu$ and $\mu'$, i.e.,
\begin{align}\label{eq:S_coupling_2}
	\pi(\cdot \times \Reals^n) = \mu(\cdot), \qquad \pi(\Reals^n \times \cdot) = \mu'(\cdot);
\end{align}
moreover, $\mu \sim \nu$ and $\mu' \sim \nu'$, where $\sim$ indicates equivalence (mutual absolute continuity) of measures. Since $\tilde{\mu}'$ has a density w.r.t.\ the Lebesgue measure, so does $\nu'$. Denoting the latter density by $q$, let us define
\begin{align*}
	\rho(z,t) \deq \E[q(X^{z,t}_T)] = \int_{\Reals^n} p(z,y;t,T) q(y)\dif y, \qquad \text{for all } (z,t) \in \Reals^n \times [0,T].
\end{align*}
Then $\rho(z,T) = q(z)$, and it follows from \eqref{eq:S_coupling_1} and \eqref{eq:S_coupling_2} that
\begin{align*}
	\rho(z,0) =  \int_{\Reals^n} p(z,y;0,T)q(y) \dif y =  \int_{\Reals^n} p(z,y;0,T) \nu'(\dif y) = \frac{\dif\mu}{\dif\nu}(z).
\end{align*}

For a nonrandom initial condition $\tilde{X}_0 = z$, consider the Gibbs measure $P^{*,z}$ with
\begin{align*}
	\frac{\dif P^{*,z}}{\dif P^z} = \frac{q(X^z_T)}{\rho(z,0)} = \frac{\exp(-H(X^z))}{\E_{P^z} \exp(-H(X^z))}
\end{align*}
corresponding to the Hamiltonian $H(X^z) = -\log q(X^z_T)$, i.e., we take $f \equiv 0$ and $g = -\log q$ in \eqref{eq:Hamiltonian}. Assuming $g$ is such that (H1) and (H2) are satisfied, Theorem~\ref{thm:main} tells us that we can obtain samples from $P^{*,z}$ using the admissible control
\begin{align*}
	u_*(x,t) = -\left(\frac{\partial v}{\partial x}(x,t)\right)^\trn
\end{align*}
with $v(x,t) = - \log \rho(x,t)$; the same control also yields the optimal solution for the random initial condition $\mu$, and in that case the law of the corresponding controlled process $(\tilde{X}^*_t ; 0 \le t \le T)$ is given by the Gibbs mixture
\begin{align}\label{eq:Gibbs_mixture_SB}
	P^{*,\mu} = \int_{\Reals^n} \mu(\dif z)P^{*,z}
\end{align}
(cf.~the discussion following the proof of Theorem~\ref{thm:main}). It is readily verified that $\tilde{X}^*_T$ has the prescribed law $\mu'$: For any bounded and measurable $h : \Reals^n \to \Reals$, we have
\begin{align*}
	\E h(\tilde{X}^*_T) &= \int_{\Reals^n \times \Reals^n} h(y) p(z,y;0,T) \frac{q(y)}{\rho(z,0)}\mu(\dif z)   \dif y \\
	&= \int_{\Reals^n \times \Reals^n} h(y) p(z,y;0,T) \nu(\dif z) \nu'(\dif y) \\
	&= \int_{\Reals^n} h(y) \mu'(\dif y).
\end{align*}
The corresponding free energies (or minimum expected costs) can be computed as follows. First, for the nonrandom initial condition $\tilde{X}_0 = z$, we have
\begin{align*}
	F(z,P^{*,z}) &= \frac{1}{2}\E \int^T_0 \big|\sigma^\trn u_*(\tilde{X}^{*,z}_t,t)\big|^2 \dif t - \E[\log q(\tilde{X}^{*,z}_T)] \\
		&= - \log \rho(z,0) \\
		&= - \log \frac{\dif \mu}{\dif \nu}(z);
\end{align*}
then, under \eqref{eq:Gibbs_mixture_SB}, 
\begin{align*}
	F(P^{*,\mu}) &= \int_{\Reals^n} F(z,P^{*,z})\mu(\dif z) \\
	&= - \int_{\Reals^n} \mu(\dif z)\log \frac{\dif\mu}{\dif \nu}(z) \\
	&= - D(\mu \| \nu),
\end{align*} 
where in the last line we have extended the definition \eqref{eq:KL} of the relative entropy $D(\cdot \| \cdot)$ to any pair of $\sigma$-finite Borel measures (recall that $\mu \sim \nu$). Using these, we can recover the following expression for the `minimum control effort' from \cite{DaiPra}:
\begin{align*}
	\frac{1}{2}\E\int^T_0 \big|\sigma^\trn u_*(\tilde{X}^*_t)\big|^2 \dif t &= \E[\log q(X^*_T)]  + F(P^{*,\mu})\\
	&= D(\mu' \| \tilde{\nu}')-D(\mu \| \nu),
\end{align*}
where 
\begin{align*}
	\tilde{\nu}'(A) \deq \int_A \left(\int_{\Reals^n} p(z,y;0,T)\nu(\dif z)\right) \dif y.
\end{align*}

In general, the determination of the Beurling--Jamison measures $\nu$ and $\nu'$ in \eqref{eq:S_coupling_1} and \eqref{eq:S_coupling_2} is not straightforward, and can be done using a forward-backward successive approximation scheme going back to the seminal work of Fortet, cf.~\cite{Essid2019} and references therein. However, an explicit solution can be given when $\mu = \delta_0$ (the Dirac measure at $z=0$) and when the reference process is a standard $n$-dimensional Brownian motion $(W_t; 0 \le t \le T)$ \cite{Follmer_1988}. In that case, $\tilde{\mu}'$ is the Gaussian measure $\gamma_T$ with mean $0$ and covariance matrix $TI_n$, $\nu = \mu = \delta_0$, and $q = \frac{\dif \mu'}{\dif \gamma_T}$. The value function is then given by
\begin{align*}
	v(x,t) = - \log \E[q(x+W_T - W_t)],
\end{align*}
and the optimal control $u_*(x,t) = - \big(\frac{\partial v}{\partial x}(x,t)\big)^\trn$ (the so-called \textit{Föllmer drift}) attains the minimum energy
\begin{align*}
	\frac{1}{2}\E\int^T_0 |u_*(\tilde{X}^*_t)|^2 \dif t = D(\mu' \| \gamma_T).
\end{align*}

\subsection{Time reversal of diffusions} A problem closely related to the Schr\"odinger bridge is the following \cite{Anderson_1982,Haussmann_Pardoux}: Consider an $n$-dimensional diffusion process
\begin{align}\label{eq:TR_forward}
	X_t = X_0 + \int^t_0 b(X_s,s)\dif s + \int^t_0 \sigma(X_s,s)\dif W_s, \qquad 0 \le t \le T
\end{align}
(where, as before, $W_t$ is an $m$-dimensional standard Brownian motion) and define its time reversal $\bar{X}_t := X_{T-t}$ for $0 \le t \le T$. The question is to determine whether $\bar{X}_t$ is itself a diffusion process, so that it can be expressed as
\begin{align}\label{eq:TR_reversal}
	\bar{X}_t = \bar{X}_0 + \int^t_0 \bar{b}(\bar{X}_s,s)\dif s + \int^t_0 \bar{\sigma}(\bar{X}_s,s)\dif\bar{W}_s
\end{align}
for some $\bar{b}$, $\bar{\sigma}$ and for some $m$-dimensional Brownian motion $\bar{W}$. In \cite{Haussmann_Pardoux}, Haussmann and Pardoux gave a solution of this problem under the following assumptions:
\begin{itemize}
	\item[(S1)] there exists a constant $K$, such that $b$ and $\sigma$ satisfy
	\begin{align*}
	&	|b(x,t)-b(\bar{x},t)| + |\sigma(x,t)-\sigma(\bar{x},t)| \le K|x-\bar{x}|, \\
	& |b(x,t)|  + |\sigma(x,t)| \le K(1+|x|)
	\end{align*}
	for all $x,\bar{x} \in \Reals^n$ and all $0 \le t \le T$;
	\item[(S2)] for almost all $t > 0$, $X_t$ has a density $p(x,t)$ satisfying a local integrability condition: For all $0 < t_0 < T$,
	\begin{align*}
		\int^T_{t_0}\int_A \left(|p(x,t)|^2 + \left|\partial_x p(x,t)\sigma(x,t)^\trn\right|^2\right) \dif x \dif t < \infty
	\end{align*}
	for all bounded open sets $A \subset \Reals^n$, where
	\begin{align*}
	\partial_x p(x,t) = (\partial_{x_1}p(x,t),\dots,\partial_{x_n}p(x,t))
	\end{align*}
	denotes the weak (distributional) derivative of $p(x,t)$ w.r.t.\ $x$.
\end{itemize}
Here, (S1) ensures that \eqref{eq:TR_forward} has a unique strong solution, while (S2) guarantees the existence of weak solutions of the forward and backward Kolmogorov equations for \eqref{eq:TR_forward}. Then there exists an $m$-dimensional Brownian motion process $\bar{W}$, such that \eqref{eq:TR_reversal} holds with
\begin{align}\label{eq:TR_b_sigma}
	\begin{split}
	&\bar{b}_i(x,t) = - b_i(x,T-t) + p(x,T-t)^{-1}\sum^n_{j=1} \frac{\partial}{\partial x_j} \big(a_{ij}(x,T-t)p(x,T-t)\big), \\
	&\bar{\sigma}_{ik}(x,t) = \sigma_{ik}(x,T-t), \qquad \bar{a}_{ij}(x,t) = a_{ij}(x,T-t)
	\end{split}
\end{align}
for $i,j = 1,\dots,n$ and $k = 1,\dots,m$. As in \cite{Haussmann_Pardoux}, we adopt the convention that the term involving $p(x,T-t)^{-1}$ is set to zero if $p(x,T-t) = 0$. We now derive the result of \cite{Haussmann_Pardoux} as a consequence of Theorem~\ref{thm:main}.

To that end, we first rewrite the drift $\bar{b}$ in \eqref{eq:TR_b_sigma} as 
\begin{align}\label{eq:TR_drift}
	\bar{b}(x,t) &= \hat{b}(x,t) + \bar{a}(x,t) \left(\frac{\partial}{\partial x}\log \bar{p}(x,t)\right)^\trn,
\end{align}
where
\begin{align}\label{eq:b_hat}
	\hat{b}_i(x,t) \deq - b_i(x,T-t) + \sum^n_{j=1} \frac{\partial}{\partial x_j}\bar{a}_{ij}(x,t), \qquad i = 1,\dots,n
\end{align}
and $\bar{p}(x,t) := p(x,T-t)$. Let $P$ denote the probability law of the process
\begin{align}\label{eq:TR_ref}
	\hat{X}_t = \hat{X}_0 + \int^t_0 \hat{b}(\hat{X}_s,s)\dif s + \int^t_0 \bar{\sigma}(\hat{X}_s,s) \dif \bar{W}_s, \qquad 0 \le t \le T
\end{align}
with $\hat{X}_0$ having density $\bar{p}(\cdot,0) \equiv p(\cdot,T)$. This will be our reference process. The processes $\{\hat{X}^{z,s}_t : s \le t \le T\}$ for $z \in \Reals^n$ and $0 \le s \le T$ are defined in the same way as before, and $P^z$ will denote the probability law of $\hat{X}^{z,0}$. We will next show that the second term on the right-hand side of \eqref{eq:TR_drift} arises as an optimal control for an appropriately defined Hamiltonian,  the controlled process now taking the form
\begin{align}\label{eq:TR_controlled}
	\tilde{X}_t = \tilde{X}_0 + \int^t_0 \Big(\hat{b}(\tilde{X}_s,s) + \bar{a}(\tilde{X}_s,s)u(\tilde{X}_s,s)\Big)\dif s + \int^t_0 \bar{\sigma}(\tilde{X}_s,s)\dif\tilde{W}_s
\end{align}
with random initialization $\tilde{X}_0$ having density $\bar{p}(\cdot,0)$.

The following technical conditions will suffice for our purposes:
\begin{itemize}
	\item[(T1)]  $X_0$ in \eqref{eq:TR_forward} has a positive density $p(\cdot,0)$ such that $x\mapsto-\log p(x,0)$ is Lipschitz-continuous and such that the inequality
	\begin{align}\label{eq:finite_moments}
		\int_{\Reals^n} \exp(c|z|^2)p(z,0)\dif z < \infty
	\end{align}
	holds with some $c > 0$;
	\item[(T2)] $b(x,t)$ in \eqref{eq:TR_forward} is three times continuously differentiable in $x$ for all $0 \le t \le T$, with bounded first, second, and third derivatives, uniformly in $t$;
	\item[(T3)] $\sigma(x,t)$ in \eqref{eq:TR_forward} is bounded, satisfies the uniform ellipticity condition \cref{eq:uniform_ellipticity}, and is four times continuously differentiable in $x$ for all $0 \le t \le T$, with bounded first, second, third, and fourth derivatives, uniformly in $t$;
\end{itemize}
Under these conditions, the density $p(0,T)$ also satisfies \cref{eq:finite_moments} \cite[Thm.~3.1]{Haussmann_Pardoux}, (R2) holds for $\hat{b}$ and $\bar{\sigma}$, and (H1)--(H2) hold for the Hamiltonian
\begin{align*}
	H(\hat{X}) = \int^T_0 f(\hat{X}_t,t)\dif t + g(\hat{X}_T)
\end{align*}
with
\begin{align}\label{eq:TR_f_g}
	f(x,t) \deq \sum^n_{i=1} \frac{\partial}{\partial x_i}\hat{b}_i(x,t), \qquad g(x) \deq - \log p(x,0).
\end{align}
Moreover, the density $p(x,t)$ of $X_t$ in \eqref{eq:TR_forward} is a classical solution of the forward Kolmogorov equation 
	\begin{align*}
		\frac{\partial}{\partial t}p(x,t) = - \sum^n_{i=1} \frac{\partial}{\partial x_i}\big(b_i(x,t)p(x,t)\big) + \frac{1}{2}\sum^n_{i,j=1} \frac{\partial^2}{\partial x_i \partial x_j}\big(a_{ij}(x,t)p(x,t)\big)
	\end{align*}
for $(x,t) \in \Reals^n \times [0,T]$ (see, e.g., \cite{Stroock2008PDEs}). It is then readily verified that the time-reversed density $\bar{p}(x,t)$ is a solution of the Cauchy problem
\begin{align*}
	\frac{\partial}{\partial t}\bar{p}(x,t) + \hat{\cL} \bar{p}(x,t) - f(x,t)\bar{p}(x,t) = 0 \text{ on } \Reals^n \times [0,T], \qquad \bar{p}(x,T) = p(x,0)
\end{align*}
where
\begin{align*}
	\hat{\cL} = \sum_i \hat{b}_i(\cdot,t) \frac{\partial}{\partial x_i} + \frac{1}{2}\sum_{i,j}\bar{a}_{ij}(\cdot,t) \frac{\partial^2}{\partial x_i \partial x_j}
\end{align*}
is the generator of the reference process \eqref{eq:TR_ref}. Since $f$ and $p(\cdot,0)$ are bounded and by virtue of the assumptions on $b$ and $\sigma$, the Feynman--Kac formula gives the following expression for $v(x,t) := -\log \bar{p}(x,t)$:
\begin{align*}
	v(x,t) &= -\log \E\exp(-H(t,x,\hat{X}^{x,t})),
\end{align*}
where, analogously to \eqref{eq:partial_H}, we have defined
\begin{align*}
	H(t,x,\hat{X}^{x,t}) := \int^T_t f(\hat{X}^{x,t}_s,s)\dif s + g(\hat{X}^{x,t}_T)
\end{align*}
with $f$ and $g$ given in \eqref{eq:TR_f_g}. Theorem~\ref{thm:main} then says that 
\begin{align*}
u_*(x,t) = - \left(\frac{\partial v}{\partial x}(x,t)\right)^\trn = \left(\frac{\partial}{\partial x}\log \bar{p}(x,t)\right)^\trn
\end{align*}
minimizes the expected cost
\begin{align*}
	J(u,z) = \tilde{\E}\left[ \int^T_0 \Bigg( \frac{1}{2} |\bar{\sigma}^\trn u(\tilde{X}_t,t)|^2 + \sum^n_{i=1}\frac{\partial}{\partial x_i}\hat{b}_i(\tilde{X}_t,t)\Bigg) \dif t   -\log p(\tilde{X}_T,0)\right]
\end{align*}
over all admissible controls for the process \eqref{eq:TR_controlled} with nonrandom initial condition $\tilde{X}_0 = z$, and also attains the Gibbs measure $\dif P^{*,z} \propto \exp(-H(\hat{X}^{z,0}))\dif P^z$. As a consequence, we obtain the following variational representation of  $-\log p(\cdot,t)$:
\begin{align*}
	&-\log p(z,t) \\
	&= \min_u \tilde{\E}\Bigg[ \int^T_{T-t} \Bigg( \frac{1}{2} |\bar{\sigma}^\trn u(\tilde{X}_s,s)|^2 + \sum^n_{i=1}\frac{\partial}{\partial x_i}\hat{b}_i(\tilde{X}_s,s)\Bigg) \dif s    - \log p(\tilde{X}_T, 0)\Bigg|\tilde{X}_{T-t} = z\Bigg],
\end{align*}
where the minimization is over all admissible controls for \eqref{eq:TR_controlled}. The second term in the integrand is the divergence of the drift vector field $\hat{b}$, which depends both on $b$ and on the derivatives of $a$ w.r.t.\ $x$, cf.~\cref{eq:b_hat}. The special case of $\sigma$ depending only on time was worked out in \cite{Pavon} and, more recently, in \cite{Berner} in the context of probabilistic generative models.

The same control $u_*$ also works for the random initial condition $\tilde{X}_0$ with density $\bar{p}(\cdot,0) \equiv p(\cdot,T)$, and the probability law of the corresponding process $\{\tilde{X}^*_t : 0 \le t \le T\}$ with $\tilde{X}^*_0$ sampled from $p(\cdot,T)$ is the Gibbs mixture
\begin{align*}
	P^{*,p(\cdot,T)} = \int_{\Reals^n} p(z,T) P^{*,z}\dif z.
\end{align*}
This is also the probability law of the process $\{\bar{X}_t : 0 \le t \le T\}$, the time reversal of $X_t$ in \eqref{eq:TR_forward}. The minimum value of the free energy in this case is given by
\begin{align*}
	F(P^{*,p(\cdot,T)}) &= \int_{\Reals^n} p(z,T) F(z,P^{*,z}) \dif z \\
	&= -\int_{\Reals^n} p(z,T)\log p(z,T)\dif z,
\end{align*}
the differential entropy of the density $p(\cdot,T)$ \cite[Sec.~1.3]{Ihara}, which is finite because $p(\cdot,T)$ has finite second moments.

\section{Conclusions}
\label{sec:conclusion}

We have revisited the work of Mitter and Newton \cite{Mitter_Newton} which used an information-theoretic interpretation of the Bayes' formula to develop an optimal control approach to sampling from conditional densities of diffusion processes. While variational formulations of the Bayes' formula as free energy minimization can be found in other works (see, e.g., Zellner \cite{Zellner} and Walker \cite{Walker}), the stochastic control formulation in \cite{Mitter_Newton} leads to an alternative procedure for conditional sampling not based on iterative methods like Markov chain Monte Carlo. We have shown that free energy minimization provides a natural framework for a number of other problems arising in the context of diffusion processes, and that the stochastic optimal control viewpoint can be used to explain the structure of the solutions to these problems. 

\section*{Acknowledgments}

This work was supported by the NSF under awards CCF-2348624 (``Towards a control framework for neural generative modeling'') and CCF-2106358
(``Analysis and Geometry of Neural Dynamical Systems''), and  by
the Illinois Institute for Data Science and Dynamical Systems (iDS${}^2$), an
NSF HDR TRIPODS institute, under award CCF-1934986.